\documentclass[11pt]{article}

\def\spacingset#1{\renewcommand{\baselinestretch}%
{#1}\small\normalsize} \spacingset{1}

\usepackage{graphicx,url,color}
\usepackage{amsmath,amsthm,amssymb}
\usepackage{multirow}
\usepackage{hhline}
\usepackage{booktabs}

\newtheorem{mythm}{Theorem}
\newtheorem{mycor}{Corollary}
\newtheorem{mydef}{Definition}
\newtheorem{mylem}{Lemma}
\newtheorem{myrem}{Remark}
\newcommand{\bern}{\text{Bernoulli}}
\newcommand{\C}{\mathcal{C}}
\newcommand{\J}{\mathcal{J}}
\newcommand{\I}{\mathcal{I}}
\newcommand{\K}{\mathcal{K}}

\newcommand{\R}{\mathbf{R}}
\newcommand{\bigO}{\mathcal{O}}
\newcommand{\eye}{I}
\newcommand{\G}{\mathcal{G}}
\newcommand{\ie}{{\it{i.e.}}}
\newcommand{\eg}{{\it{e.g.}}}
\newcommand{\nset}{\left[n\right]}
\newcommand{\kset}{\left[k\right]}

\renewcommand{\O}{\mathbf{O}}

\renewcommand{\Pr}{\text{Pr}}

\DeclareMathOperator{\diag}{diag}

\DeclareMathOperator{\tr}{Tr}

\DeclareMathOperator{\range}{range}
\DeclareMathOperator*{\argmax}{arg\,max}
\newcommand{\ocs}{\text{OCS}}

\newcommand{\lmat}{\left[}
\newcommand{\rmat}{\right]}

\usepackage[table]{xcolor}

\title{Robust and efficient multi-way spectral clustering}
\author{Anil Damle\thanks{Department of Mathematics, University of California, Berkeley (damle@berkeley.edu)}, Victor Minden\thanks{Institute for Computational \& Mathematical Engineering, Stanford University (vminden@stanford.edu)}, and Lexing Ying\thanks{Department of Mathematics and Institute for Computational \& Mathematical Engineering, Stanford University (lexing@stanford.edu)}}

\begin{document}



\maketitle

\begin{abstract}
We present a new algorithm for spectral clustering based on a column-pivoted QR factorization that may be directly used for cluster assignment or to provide an initial guess for \texttt{k-means}. Our algorithm is simple to implement, direct, and requires no initial guess. Furthermore, it scales linearly in the number of nodes of the graph and a randomized variant provides significant computational gains. Provided the subspace spanned by the eigenvectors used for clustering contains a basis that resembles the set of indicator vectors on the clusters, we prove that both our deterministic and randomized algorithms recover a basis close to the indicators in Frobenius norm. We also experimentally demonstrate that the performance of our algorithm tracks recent information theoretic bounds for exact recovery in the stochastic block model. Finally, we explore the performance of our algorithm when applied to a real world graph.
\end{abstract}

Spectral clustering has found extensive use as a mechanism for detecting well-connected subgraphs of a network. Typically, this procedure involves computing an appropriate number of eigenvectors of the (normalized) Laplacian and subsequently applying a clustering algorithm to the embedding of the nodes defined by the eigenvectors. Currently, one of the most popular algorithms is \texttt{k-means++} \cite{kmeanspp}, the standard iterative \texttt{k-means} algorithm \cite{lloyd} applied to an initial clustering chosen via a specified random sampling procedure. Due to the non-convex nature of the \texttt{k-means} objective, however, this initialization does not preclude convergence to local minima, which can be poor clusterings. 

We provide an alternative, direct (non-iterative) procedure for clustering the nodes in their eigenvector embedding. It is important to note that our procedure is not a substitute for \texttt{k-means++} when tackling general (\emph{i.e.}, non-spectral) clustering problems. For spectral embeddings of graphs with community structure, however, we take advantage of additional geometric structure of the embedding to build a more robust clustering procedure. Furthermore, our algorithm is built out of a simple column-pivoted QR factorization, making it easy to implement and use. Finally, a simple randomized acceleration of our algorithm substantially reduces the cost of cluster assignment, making it feasible for large problems of practical interest.

\section{Background and setup}

Given a simple undirected graph $\G$ with adjacency matrix ${A\in\{0,1\}^{n\times n}}$, we consider the multi-way clustering problem of partitioning the vertices of $\G$ into $k$ disjoint clusters.  A common (albeit unrealistic) generative model for graphs exhibiting this sort of cluster structure is is the \emph{$k$-way stochastic block model}.

\begin{mydef}[Stochastic block model \cite{sbm}]
Partition $[n]$ into $k$ mutually disjoint and non-empty clusters $\C_1,\dots,\C_k$.  Given probabilities $p$ and $q$ such that $p>q$, let $M \in [0,1]^{n\times n}$ have entries $M_{ii}\equiv0$ and, for $i\ne j$,
\begin{align*}
M_{ij} &\equiv \left\{\begin{array}{ll}
p,& \{i,j\}\subset\C_k \text { for some } k,\\
q,& \text{else.}\end{array}\right.
\end{align*}

A symmetric adjacency matrix $A\in\{0,1\}^{n\times n}$ with ${A_{ij} \sim \bern(M_{ij})}$ for $i< j$ and $A_{ii}\equiv0$ for all $i$ is said to be distributed according to the $k$-way stochastic block model (SBM) with clusters $\{\C_i\}_{i=1}^k$, within-cluster probability $p$, and between-cluster probability $q$.
\end{mydef}

For an SBM with equisized clusters, the maximum-likelihood estimate for the clusters can be phrased in terms of maximizing the number of within-cluster edges; that is, given $A$, find a matrix $X$ whose columns are indicator vectors for cluster membership such that $X$ attains the optimal value of the combinatorial optimization problem
\begin{equation}\label{eq:combo}
\begin{aligned}
& \underset{X}{\text{maximize}}
& & \tr\left(X^TAX\right) \\
& \text{subject to}
& & X \in \{0,1\}^{n\times k},\quad X^TX = \frac{n}{k}\eye_{k}.
\end{aligned}
\end{equation}
If $A$ is not assumed to be a random sample from the SBM, then the above problem does not have the interpretation of maximum-likelihood estimation, though it remains a common starting point for clustering. Given that the combinatorial optimization problem is NP-hard, it is typical to relax \eqref{eq:combo} to a computationally-tractable convex formulation.  

A common relaxation of \eqref{eq:combo} is to remove the restriction that $X\in\{0,1\}^{n\times k}$ and instead optimize over real-valued matrices,
\begin{equation}\label{eq:spectral}
\begin{aligned}
& \underset{X}{\text{maximize}}
& & \tr\left(X^TAX\right) \\
& \text{subject to}
& & X \in \R^{n\times k}, \quad X^TX = \frac{n}{k}\eye_{k}.
\end{aligned}
\end{equation}
While this optimization problem is still non-convex, it follows from the Courant-Fischer-Weyl min-max principle that an optimal point $X_*$ is given as $X_* = V_kQ$, where $V_k\in\R^{n\times k}$ contains the eigenvectors of $A$ corresponding to the $k$ largest eigenvalues and $Q\in\O_k$ is an arbitrary orthogonal transformation.  Because the solution $X_*$ is no longer discrete, the canonical spectral clustering approach uses the rows of $X_*$ as coordinates in a standard point-cloud clustering procedure such as \texttt{k-means}.

We propose an algorithm based on a column-pivoted QR factorization of the matrix $V_k^T$ that can be used either as a stand-alone clustering algorithm or to initialize iterative algorithms such as \texttt{k-means}.  Our approach stems from the computational quantum chemistry literature, where the same basic procedure has been developed as a method for constructing spatially-localized bases of an eigenspace associated with the nonlinear eigenproblem of Kohn-Sham density functional theory \cite{damle2015,damle2016}.  Our numerical experiments show that our approach closely tracks the information-theoretic limit in terms of exact clustering of the stochastic block model, whereas \texttt{k-means++} on its own does not.

\subsection{Related work}
An algorithm using a column-pivoted QR for cluster assignment of general point-cloud data\textemdash with an assumption of orthogonality amongst the clusters\textemdash has previously appeared in the literature in the context of spectral relaxation of the \texttt{k-means} objective \cite{zha2001spectral}.  Curiously, though we find the basic idea to be powerful, this algorithmic approach seems to have been ignored and we can find no further reference to it.  We expand upon this work, taking advantage of algorithmic improvements appearing in the computational chemistry literature for greater efficiency.  Further, by addressing the sparse adjacency matrix case instead of the Gram matrix of point-cloud data, we are able to strongly motivate our approach based on proven spectral properties of model graphs and graphs with assumed underlying cluster structure. 

Spectral methods as we discuss here stem from work on the Fiedler vector \cite{fiedler1973algebraic,donath1973lower} for the two-block case and spectral embeddings coupled with \texttt{k-means} clustering \cite{Ng01onspectral} in the multi-block case. For a more comprehensive overview of initialization procedures for \texttt{k-means} see, \emph{e.g.}~, Celebi et al.~\cite{celebi2013comparative}. Recent analysis of these methods applied to the SBM \cite{rohe,gao2015achieving} focuses on misclassification of nodes. Another line of work considers matrices besides the adjacency matrix or normalized Laplacian to achieve theoretical detection thresholds \cite{krzakala2013spectral,massoulie2014community}.  Other recent work \cite{Montanari_Sen} demonstrates where spectral methods break down and argues for the use of SDP-based methods. 

Recently, there has been significant work on understanding when it is possible to exactly recover communities in the sparse SBM, wherein the probability of connections between nodes is $\Theta(\log{n}/n)$. Specifically, the information theoretic bound for when exact recovery is possible with two blocks was developed by Abbe et al. \cite{abbe} and an SDP-based algorithm achieving the bound was proposed. Additional recent work \cite{abbe2015community,hajek2015achieving,hajek2016achieving} has extended this theory to the multi-block case (for a constant or slowly growing number of clusters) and generalized the SDP approach. We will return to this example in the numerical results section.

\section{Algorithm}

\subsection{Preliminaries}

Suppose $\G$ is a simple undirected graph with symmetric adjacency matrix $A\in\{0,1\}^{n\times n}$. Let the eigendecomposition of $A$ be given by
\begin{align*}
A &= \lmat\begin{array}{c|c}V_k & V_{n-k}  \end{array} \rmat\lmat\begin{array}{c|c} \Lambda_k & 0\\\hline 0 & \Lambda_{n-k}\end{array}\rmat \lmat\begin{array}{c|c}V_k & V_{n-k}  \end{array} \rmat^T,
\end{align*}
where $V_k\in\R^{n\times k}$ and $V_{n-k}\in\R^{n\times(n-k)}$ contain pairwise orthonormal columns of eigenvectors, and $\Lambda_k = \diag(\lambda_1,\lambda_2,\dots,\lambda_k)$ contains the $k$ largest eigenvalues of $A$ sorted as $\lambda_1 \ge \lambda_2,\ge \dots\ge\lambda_k$.  Of interest in spectral clustering is the structure of the eigenspace spanned by the columns of $V_k$ or a related matrix, as we make explicit below.
\begin{mydef}
Suppose $\C_1,\C_2,\dots,\C_k$ is a clustering of the vertex set $[n]$.  We say the matrix $W'\in \{0,1\}^{n\times k}$ with entries given by
\begin{align*}
W'_{ij} &\equiv \left\{\begin{array}{cl}1, & \text{if vertex $i$ belongs to cluster $j$,}\\ 0, &\text{otherwise,} \end{array} \right.
\end{align*}
is an \emph{indicator matrix} for the underlying clustering.
If $W'$ is an indicator matrix and $W\in \R_+^{n\times k}$ is given by scaling the columns of $W'$ to have unit $\ell_2$ norm then we call $W$ a \emph{normalized indicator matrix} for the underlying clustering.
\end{mydef}

Suppose $A$ is sampled from the $k$-way SBM on $n$ nodes with within- and between-cluster connection probabilities $p$ and $q$. Letting $W$ be a normalized indicator matrix for the underlying clustering and provided that $p$ and $q$ do not decay too quickly asymptotically, standard matrix concentration inequalities \cite{tropp,sharp_bounds} coupled with the Davis-Kahan theorem \cite{daviskahan} give convergence of the spectral projector $V_kV_k^T$ in $n$, \ie, for fixed $k$ there exists $\delta_n=o(1)$ such that
\begin{align*}
\left\|WW^T - V_kV_k^T\right\|_2 \leq \delta_n
\end{align*}
with high probability.
The exact nature of these concentration results depends strongly on the asymptotic behavior of $p$ and $q$, however, the general idea is that with high probability any given $A$ will be close to the average adjacency matrix. More generally, if $A$ is a symmetric adjacency matrix not necessarily coming from the SBM that is sufficiently close to $M$, the same basic argument applies \cite{luxburg}. Therefore, let $\epsilon > 0$ be such that
\begin{align*}
\left\|WW^T - V_kV_k^T\right\|_2 \le \epsilon.
\end{align*}
From this bound on the distance between spectral projectors we can conclude the existence of $Q\in\O_k$ such that
\begin{align}\label{eq:rotnorm}
\left\|W - V_kQ\right\|_F \le \sqrt{k}\epsilon + \bigO(\epsilon^2).
\end{align}

Concisely, we expect the eigenvector matrix $V_k$ to look like the normalized indicator matrix $W$, under a suitable orthogonal transformation.  Alternatively, $V_kQ$ provides an approximately sparse basis for the range of $V_k$.  Thus, while any orthogonal transformation of the (appropriately scaled) eigenvectors is a solution to \eqref{eq:spectral}, there exists a specific solution that directly reflects the underlying cluster structure. If $W$ and $V_k$ in \eqref{eq:rotnorm} were both known, finding an orthogonal transformation to make $V_k$ resemble $W$ is not a difficult problem.  However, we seek to find this transformation without knowing the clustering $\emph{\emph{a priori}}.$

In general, when handed a graph that may not come from the SBM, we cannot leverage explicit knowledge of eigenvectors. However, for graphs that admit a $k$-way clustering, there is a slightly more general geometric structure we can appeal to\textemdash so-called orthogonal cone structure (OCS). Motivation for this model and the nomenclature is borrowed from Schiebinger et al.~\cite{schiebinger2015geometry}. However, they address a slightly different (though related) problem and we do not use their model exactly. Additional justification for finding OCS in graphs with community structure may be found in Gharan \& Trevisan \cite{gharan2014partitioning} and Benzi et al.~ \cite{benzi2013decay}. 

To aid in our discussion, for a given orthonormal vector $q \in \R^k$ and scalar $\mu\in[0,1]$ we define the cone $\K_{\mu}(q)$ as
\[
\K_{\mu}(q) \equiv \left\{x \in \R^k \,\left\vert\, \frac{x^Tq}{\|x\|_2} \geq 1-\mu \right. \right\}.
\] 
Given this simple definition of a cone, we may now define OCS rigorously.
\begin{mydef}
A set of points $\left\{x_i\right\}_{i=1}^n$ in $\R^k$ exhibit \emph{orthogonal cone structure} (OCS) with parameters $\eta\in(0,1],$ $\mu\in\left[0,1\right],$ and $\delta > 0$ if the following hold:
\begin{enumerate}
\item There exists a set of orthonormal vectors $\left\{q_i\right\}_{i=1}^k$ such that at least $\eta n$ of the $x_i$ satisfy  
\begin{equation}
\label{eq:in_cone}
x_i \in \K_{\mu}\left(q_j\right)
\end{equation}
for exactly one $j = 1,\ldots,k.$ Furthermore, each of the $\K_{\mu}\left(q_j\right)$ contains at least one $x_i.$
\item Let $\I$ denote the set of $x_i$ that satisfy~\ref{eq:in_cone} for some $j,$ then 
\[
\|x_i\|_2 \leq \delta \min_{l\in\kset}\max_{x_j\in\K_{\mu}(q_l)}\|x_j\|_2 \qquad \forall \; x_i\notin\I.
\]
\end{enumerate}
When these two conditions are satisfied for a set of points in $\R^k$ we say $\left\{x_i\right\}_{i=1}^n \in \ocs_k(\eta,\mu,\delta).$ 
\end{mydef}
This definition essentially says that given $n$ points, a fraction $\eta$ of them can be assigned uniquely to a set of cones with orthogonal centers, while the remaining points have small norm relative to the largest point in each cone. Importantly, given the cone centers $\left\{q_j\right\}_{j=1}^k$ it is easy to correctly cluster the points that lie in the cones\textemdash under the assumption that the cones reflect the desired cluster structure. This is accomplished by simply checking which $q_j$ a given point has largest magnitude inner product with. Those points that lie outside the cones are assumed to not be particularly well-suited to any of the clusters and any assignment is considered acceptable.

In both OCS and the SBM we have $n$ points contained in $k$ cones, the centers of which are assumed to be orthogonal. However, they are represented in an arbitrary coordinate system using the eigenvectors, which means the embedding coordinates of each point do not make it easy to determine which points belong to which cone. If we can find a coordinate system roughly aligned with the centers of the cones, we can then rotate the points into this system and read off the cluster assignment based off the largest magnitude entry. This idea is illustrated in Figure~\ref{fig:rotate_cones}.

\begin{figure}[ht!]
\centering
\includegraphics[width=.7\linewidth]{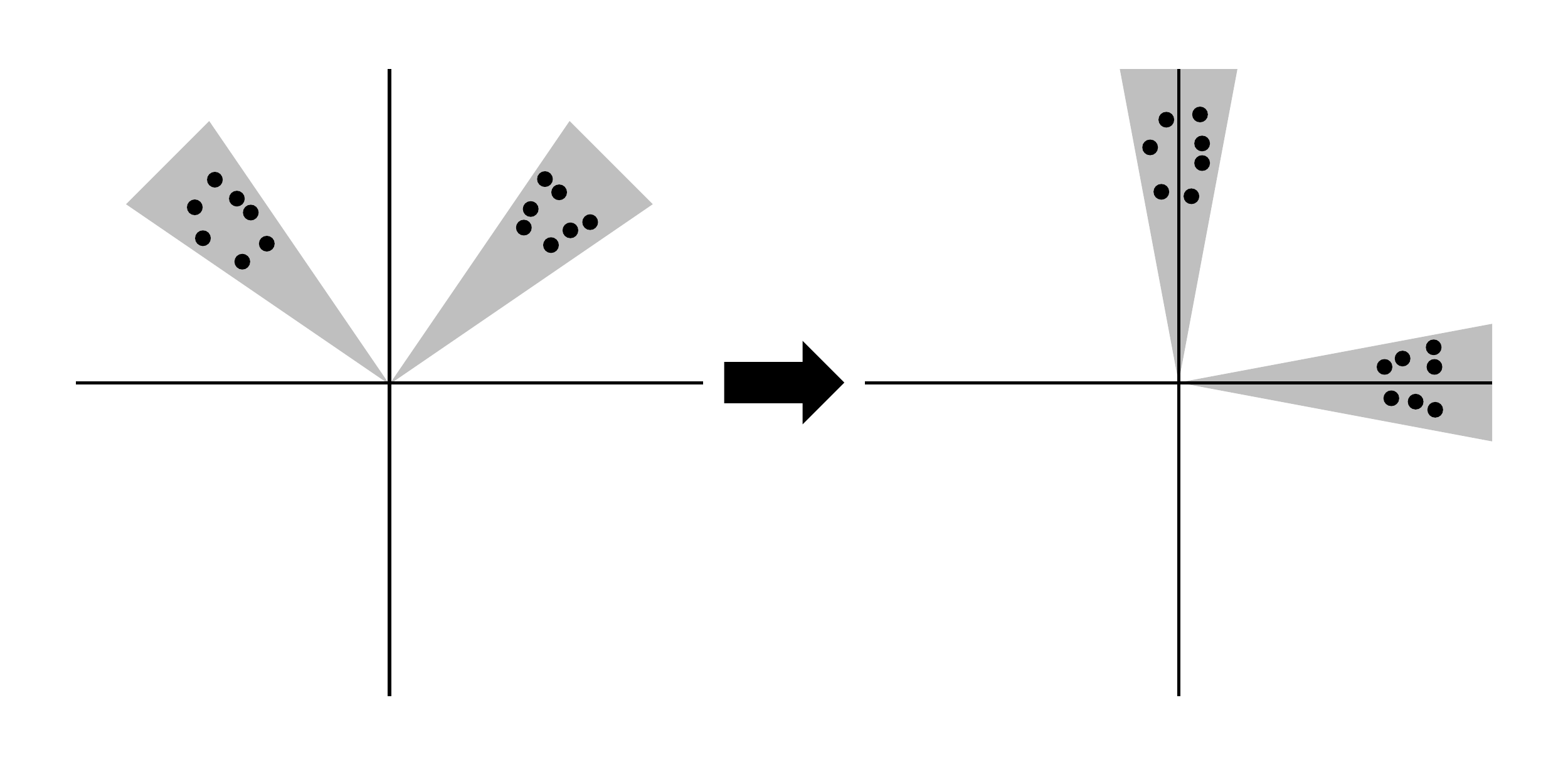}
\caption{Two dimensional demonstration of the expected orthogonal cone structure. While the computed eigenvectors (left) may not readily allow for coordinate-based clustering, there is a rotation of the coordinate system (right) that does.}
\label{fig:rotate_cones}
\end{figure}

We accomplish this task primarily with the use of a column-pivoted QR factorization. We define this factorization notationally for matrices that are wider than they are tall, which is the setting we will require later.
\begin{mydef}
Let $B\in\R^{k\times n}$ with $k\le n$ have rank $k$.
The \emph{column-pivoted QR factorization} (CPQR) of $B$ is the factorization
\[
B\Pi = Q\begin{bmatrix} R_1 & R_2 \end{bmatrix}
\]
as computed via the Golub-Businger algorithm \cite{businger1965linear} (see also Chapter 5 of Golub \& Van Loan \cite{golub}) where $\Pi \in \R^{n\times n}$ is a permutation matrix, $Q \in \O_k$ is an orthogonal matrix, $R_1 \in \R^{k\times k}$ is upper triangular, and $R_2 \in \R^{k\times (n-k)}$.
\end{mydef}

Ideally, the permutation $\Pi$ is chosen to make $R_1$ as well-conditioned as possible given $B$.  The Golub-Businger algorithm chooses this permutation using a greedy heuristic wherein at each step of the algorithm the column with largest norm is picked as the pivot and the remaining columns are orthogonalized against it.

Lastly, we briefly define a slightly less common matrix factorization we leverage in our algorithm\textemdash the polar factorization
\begin{mydef}
For a matrix $B\in\R^{k\times k},$ the \emph{polar factorization} of $B$ is the factorization
\[
B = UH
\]
where $U\in\R^{k\times k}$ is orthogonal and $H\in\R^{k\times k}$ is positive semi-definite.
\end{mydef}
Such a factorization always exists and, if $B$ has full rank, $H$ is guaranteed to be positive definite. Importantly, in any orthogonally-invariant norm $U$ is the closest orthogonal matrix to $B$ \cite{FanHoffman}. Finally, the polar factorization may be computed easily from the orthogonal factors from the singular value decomposition (SVD) of $B$ with computational cost $\bigO(k^3)$ or by a specialized routine. Further details may be found in, \eg, Higham \cite{Higham_polar}.

\subsection{Algorithm statement}
\label{sec:alg_sub}
Given $V_k$ we propose the following algorithm for finding $k$ clusters in the graph:
\begin{enumerate}
\item Compute the CPQR factorization $$V_k^T\Pi = QR.$$
\item Let $\C$ denote the $k$ columns selected as the first $k$ columns of $\Pi.$
\item Compute the polar factorization $\left(V_k^T\right)_{:,\C} = UH.$
\item For each $j\in \left[n \right]$ assign node $j$ to cluster $$c_j \equiv \argmax_i  \left(\left\lvert U^TV_k^T \right\rvert_{i,j} \right). $$
\end{enumerate}
Note that $\lvert \cdot \rvert$ is used to represent the elementwise absolute value operation when applied to a matrix and here we assume that every node belongs to exactly one cluster.  For real-world applications, if all entries in a column are small we may optionally not assign the corresponding node to any cluster, but we do not explore that further at this time.

The above algorithm is strikingly simple, there is no dependence on any sort of initial guess, it is naturally invariant to the data ordering, and it is constructed out of common and efficient matrix factorizations. From an implementation standpoint the CPQR may be computed using, \emph{e.g.}, the \texttt{qr} function in MATLAB\textsuperscript{\textregistered} or the LAPACK \cite{laug} routine \texttt{DGEQP3} \cite{blas3QRCP}.  Overall, the cost of this algorithm is linear in the number of nodes of the graph and quadratic in the number of clusters, \emph{i.e.}, $\bigO(nk^2)$. The lack of an initial guess is a particularly important distinction from the generally used \texttt{k-means} algorithm.

A simple randomized variant of our algorithm allows us to more efficiently compute a matrix $Q$ such that $Q^TV_k$ may be used for cluster assignment with high probability. To do this, we define the probability distribution $\rho$ over $\left[ n \right]$ as
\[
\Pr(\{j\}) = \frac{\left\| \left(V_k^T\right)_{:,j}\right\|_2}{k}.
\]
Sampling strategies based on $\rho$ are not new and corresponds to sampling via the so-called leverage scores \cite{mahoney2009cur} of the spectral projector $V_kV_k^T$. Generically, similar strategies are common in the construction of low rank factorizations \cite{frieze2004fast}. However, here we have a different goal. The spectral projector is always exactly rank $k$, and we are not concerned with sufficiently sampling its range. Rather, we simply need to ensure sampling of one column associated with each cluster for the QR factorization to be effective; we are essentially throwing away excess data. The use of leverage scores ensures that the probabilities of selecting a column from a given cluster are all (roughly) equal.

Given $V_k$, an oversampling factor $\gamma$, and a failure probability $\delta$, the following algorithm computes the cluster assignments:
\begin{enumerate}
\item Sample a set $\J$ of $\gamma k \log\left( \frac{k}{\delta}\right) $ nodes from $\left[ n \right]$ based on the discrete distribution $\rho$.
\item Compute the CPQR factorization $$\left(V_k^T\right)_{:,\J}\Pi = QR.$$
\item Let $\C$ denote the $k$ columns selected as the first $k$ columns of $\Pi.$
\item Compute the polar factorization $\left(\left(V_k^T\right)_{:,\J}\right)_{:,\C} = UH.$
\item For each $j\in \left[n \right]$ assign node $j$ to cluster $$c_j \equiv \argmax_i  \left(\left\lvert U^TV_k^T \right\rvert_{i,j} \right).$$
\end{enumerate}

The cost of the CPQR is now $\bigO(k^3\log k),$ yielding a cluster assignment operator in time independent of $n$. While computing $U^TV_k$ in step 3 formally scales as $\bigO(nk^2)$, the dominant cost in the deterministic algorithm is the CPQR. If one only wishes to cluster a subset of the nodes or check if two nodes are in the same cluster, then $U^T$ need only be applied to some columns of $V_k^T$. Furthermore, $U^TV_k^T$ scales well when computed on parallel architectures and the assumed structure, specifically the locality, of the product could also potentially be used to reduce the cost of its computation. The other dependence on $n$ in the form of $\bigO(nk)$ is in computing the column norms of $V_k^T$ to build $\rho$. This operation also parallelizes well and asymptotically costs no more than writing down $V_k$ itself.  To truly make the cost of constructing $U$ independent of $n$, however, one may sample nodes uniformly to build $\J$. In this case one must sample slightly more than $k\log k$, in a manner that depends on the size of the smallest cluster relative to $n$, to ensure the same failure probability.

Prior to embarking on the proofs, we provide a brief high-level rationale for the use of a CPQR. Assuming every point belongs to exactly one cluster, the \texttt{k-means} algorithm applied to the spectral embedding may be expected to work in principle given a sufficiently close initial guess for the cluster centers. Equivalently, given one node from each cluster we may use it as a proxy for a cone center. It turns out the CPQR procedure achieves this goal when we are working with the spectral embedding for a graph with community structure.

Heuristically, consider picking an initial node at random.  It belongs to one of the clusters and thus might be a reasonable surrogate for a cone center that represents that cluster. Now, we need a procedure to pick a point that is guaranteed to be in a different cluster to use as another center. For general clustering problems this may be difficult. However, here we observed that the clusters have additional structure: we expect them to be orthogonal. This motivates the use of the CPQR. Once we orthogonalize all of the points against the center point of the first cluster, the norms of other points in the same cluster become small while the norms of points in different clusters remain more or less the same. Therefore, picking the point with largest remaining norm should yield a viable center for a different cluster. By repeating this procedure $k$ times we hope to get a good set of $k$ different cluster centers. This procedure is codified by the CPQR factorization introduced earlier.

The preceding procedure ideally gives us $k$ points, each of which represents one of the underlying cones. However, we expect the cones to exhibit orthogonal structure and the selected cluster centers may not be quite orthogonal. Hence, we compute the ``closest'' orthogonal matrix to $\left(V_k^T\right)_{:,\C}$ via the polar factorization. This orthogonal matrix defines a coordinate system that aligns with the expected structure of the points. Once viewed in this coordinate system (by applying the appropriate rotation) we may simply cluster the points based on their largest magnitude entry.

\begin{myrem}
One may, alternatively, simply use the $Q$ factor from the CPQR as the new coordinate system, which corresponds to a greedy orthogonalization of the $k$ points. In the later examples, because $R_1$ ends up being quite well conditioned there is not a significant difference between these two strategies. The polar factorization has the potential advantage of not giving preferential treatment to any of coordinate directions based on the order in which they were selected.
\end{myrem}

\subsection{Analysis}
\label{subsec:analysis}
To make our analysis concrete, we assume the indicator structure inherent to the SBM. The justification for OCS is the same. In particular, one may relax the model for $W$ to represent more general OCS. As long as $W$ is assumed to only contain one nonzero per row the following analysis still holds (with slightly larger constants to account for the fact that each column of $W$ may contain more than one nonzero value). This is reasonable because the problem is fundamentally the same for both the SBM and more general OCS: given a collection of orthogonal indicator vectors that have been arbitrarily rotated and perturbed, how do we recover the underlying structure? 

Theorem 7.2 of Gu \& Eisenstat \cite{gu1996efficient} proves that for a CPQR as defined above
\begin{align}\label{eq:fnk}
\sigma_{k}\left(R_{1}\right) \geq \frac{\sigma_k(B)}{f(n,k)},
\end{align}
where $\sigma_j$ denotes the $j^\text{th}$ singular value of a matrix and $f(n,k) = 2^k\sqrt{n}$. This bound ensures that the singular values of $R_1$ are not too much smaller than the singular values of $B$.  The CPQR defined here is related to a broader class of rank-revealing QR (RRQR) factorizations \cite{chandrasekaran1994rank,gu1996efficient} that seek a permutation $\Pi$ such that \eqref{eq:fnk} holds for different forms of the denominator $f(n,k)$, \emph{e.g.}, $\sqrt{k(n-k)+\min(n,n-k)}$. Because the matrices we deal with are quite far from the pathological worst-case examples we find the standard CPQR algorithm sufficient for our needs.

We will show our algorithm finds a matrix $U$ such that $V_kU$ is close to $W$ (up to a signed column permutation) without explicit knowledge of the sparsity structure of $W.$ The potentially different column order inherited from the signed permutation simply corresponds to a relabeling of the clusters and since we take use the largest entry in absolute value to do cluster assignment the sign does not matter.

We prove two lemmas that ensure the CPQR pivoting strategy applied to $V_k^T$ (or a randomly selected subset of its columns) implicitly identifies $k$ linearly independent columns of $WW^T$.
\begin{mylem}
\label{lem:piv}
Let $W \in \R^{n\times k}$ and $V \in \R^{n\times k}$ have orthonormal columns and satisfy
\[
\left\|WW^T - VV^T\right\|_2 \leq \epsilon
\]
with $\epsilon < 2^{-k}/\sqrt{n}$. If $V^T\Pi = QR$ is a CPQR factorization and $\C$ denotes the original indices of the first $k$ columns selected by the permutation $\Pi$, then
\[
\range\left\{\left(WW^T\right)_{:,\C} \right\} = \range\left\{W\right\}.
\]
\end{mylem}

\begin{proof}
Theorem 7.2 of Gu \& Eisenstat \cite{gu1996efficient} coupled with the fact that $\sigma_i(V)= 1$ for $i=1,\dots,k$ implies that $\sigma_{\min}(R_1) \geq 2^{-k} /\sqrt{n}$. Substituting in the CPQR factorization yields
\[
\left(WW^T + E\right)\Pi = VQ\begin{bmatrix}R_1 & R_2 \end{bmatrix}
\]
with $\|E\|_2 \leq \epsilon$.  Now, $\sigma_{\min}(R_1) \geq \epsilon$ implies that $\left(W^T\right)_{:,\C}$ is non-singular, as otherwise the distance from $VQR_1$ to the nearest singular matrix is less than $\epsilon$, which is a contradiction.  The result follows.
\qquad
\end{proof}

\begin{mylem}
\label{lem:piv_rand}
Let $W \in \R^{n\times k}$ be a normalized indicator matrix and $V \in \R^{n\times k}$ have orthonormal columns and satisfy
\[
\left\|WW^T - VV^T\right\|_2 \leq \epsilon
\]
with $\epsilon < 2^{-k}/\sqrt{n}$ and assume there exists $\gamma > 0$ such that
\[
\frac{1}{\gamma k} \leq \sum_{i\in\C_j}\left\|V_{:,i}\right\|^2_2
\]
for $j\in\kset$. Let $\J$ denote $\gamma k \log \frac{k}{\delta}$ samples with replacement from the discrete distribution $\rho$ over $\left[ n \right]$. If $\left(V^T\right)_{:,\J}\Pi = QR$ is a CPQR factorization and $\C$ denotes the original indices of the first $k$ columns selected by the permutation $\Pi$, then with probability $1-\delta$
\[
\range\left\{\left(WW^T\right)_{:,\C} \right\} = \range\left\{W\right\}.
\]
\end{mylem}

\begin{proof}
The matrix $WW^T$ has $k$ distinct linearly independent columns, and because of the normalization the sum of square column norms for each set is $1$. Therefore, the distribution $\rho$ has mass $\frac{1}{k}$ on each set of linearly independent columns. Now, a simple coupon-collecting argument \cite{motwani2010randomized} over $k$ bins where the smallest mass a bin can have is $1 / \left(\gamma k \right)$ ensures that $\J$ contains $k$ linearly independent columns. The result now follows immediately from Lemma~\ref{lem:piv}.
\qquad
\end{proof}

\begin{myrem}
In either lemma, one may substitute the CPQR for a different RRQR \cite{chandrasekaran1994rank} or strong RRQR \cite{gu1996efficient} and inherit the appropriate upper bound on $\epsilon$ based on the corresponding $f(n,k)$.
\end{myrem}

By assumption, $W\approx VZ$ for some unknown orthogonal matrix $Z.$ The key aspect of Lemma~\ref{lem:piv} and Lemma~\ref{lem:piv_rand} is that the CPQR identifies sufficient information to allow the construction of an orthogonal matrix $U$ that is close to $Z.$

\begin{mythm}
\label{thm:main}
Let $W \in \R^{n\times k}$ be a normalized indicator matrix and suppose $n>4$. Let $V \in \R^{n\times k}$ with orthonormal columns satisfy
\[
\left\|WW^T - VV^T\right\|_2 \leq \epsilon
\]
with $\epsilon < 2^{-k} /\sqrt{n}$.  If $U$ is computed by the deterministic algorithm in subsection \ref{sec:alg_sub} applied to $V^T$, then there exists a permutation matrix $\widehat{\Pi}$ such that
\[
\| W\widehat{\Pi} - VU \|_F \leq \epsilon k \sqrt{n} \left(2+2\sqrt{2}\right) +\mathcal{O}\left(\epsilon^2\right).
\]
\end{mythm}

\begin{proof}
Let $\C$ denote the original indices of the first $k$ columns selected by the permutation $\Pi$. Based on the nonzero structure of $W$ and Lemma~\ref{lem:piv}, there exists a permutation matrix $\widehat{\Pi}$ such that $R_I\equiv \widehat{\Pi}^T\left(W^T\right)_{:,\C}$ is diagonal with positive diagonal entries. Our assumptions on $W$ and $V$ imply an that there exists an orthogonal matrix $Z$ such that 
\[
\left\|Z\widehat{\Pi}R_I - \left(V^T\right)_{:,\C}\right\|_F \leq \sqrt{k}\epsilon + \mathcal{O}\left(\epsilon^2\right)
\]
Now, observe that $\left(Z\hat{\Pi}\right)R_I$ is the polar factorization of a matrix. Using perturbation analysis for polar factorizations \cite{Higham_polar} yields
\[
\left\|Z\widehat{\Pi} - U\right\|_F \leq \epsilon\left(1+\sqrt{2}\right)k\sqrt{n} + \mathcal{O}\left(\epsilon^2\right).
\]  
Similar to above, our assumptions also imply that 
 \[
\left\|W\widehat{\Pi} - VZ\widehat{\Pi}\right\|_F \leq \sqrt{k}\epsilon + \mathcal{O}\left(\epsilon^2\right).
\]
Substituting $U$ for $Z\widehat{\Pi}$ and using the bound of their difference allows us to conclude the desired result.
\qquad
\end{proof}

\begin{mythm}
\label{thm:main_rand}
Let $W \in \R^{n\times k}$ be a normalized indicator matrix and suppose $n>4$. Let $V \in \R^{n\times k}$ with orthonormal columns satisfy
\[
\left\|WW^T - VV^T\right\|_2 \leq \epsilon
\]
with $\epsilon < 2^{-k} /\sqrt{n}$ and assume there exists $\gamma > 0$ such that
\[
\frac{1}{\gamma k} \leq \sum_{i\in\C_j}\left\|V_{:,i}\right\|^2_2
\]
for $j\in\kset$. If $U$ is computed following the randomized algorithm in subsection \ref{sec:alg_sub} then with probability $1-\delta$ there exists a permutation matrix $\widehat{\Pi}$ such that
\[
\| W\widehat{\Pi} - VU \|_F \leq \epsilon k \sqrt{n} \left(2+2\sqrt{2}\right) +\mathcal{O}\left(\epsilon^2\right).
\]
\end{mythm}
\begin{proof}
The proof mirrors that of Theorem~\ref{thm:main}, we simply use Lemma~\ref{lem:piv_rand} in place of Lemma~\ref{lem:piv} to ensure that the permutation $\widehat{\Pi}$ exists.
\qquad
\end{proof}

Theorems~\ref{thm:main} and~\ref{thm:main_rand} showt that if the subspace spanned by $V_k$ is close to one containing indicator vectors on the clusters, then both of the proposed algorithms yield an orthogonal transformation of $V_k$ approximating the cluster indicators. This representation of the subspace may then be directly used for cluster assignment. If desired, one may also use the computed clustering to seed the \texttt{k-means} algorithm and attempt to further improve the results with respect to the \texttt{k-means} objective function. This can either be accomplished by computing the cluster centers, or by using the points associated with the set $\C$ as the initial cluster centers.

\subsection{Connections to the SBM}
The preceding theory is rather general, as it only relies on normalized indicator structure of $W.$ The connection to the SBM is by virtue of $W$ representing the range of $\mathbb{E}A = M$, which is a rank $k$ matrix. Specifically, for a given instance of the SBM, Corollary 8.1.11 of Golub \& Van Loan \cite{golub} yields 
\begin{equation}
\label{eq:proj_error}
\|WW^T - V_kV_k^T\|_2 \leq \frac{4}{m(p-q)-p}\left\|\left(A-M\right)\right\|_2.
\end{equation}
We now restrict ourselves to the regime where $p,q\sim \frac{\log n}{n},$ and introduce the parametrization $p \equiv \alpha \frac{\log m}{m}$ and $q \equiv \beta \frac{\log m}{m}$ as in Abbe et al.~ \cite{abbe}. Note that here we assume a constant number of clusters $k$ and therefore $n\rightarrow\infty$ implies that $m\rightarrow\infty.$  In this regime, $\|A - M\|_2 \lesssim \sqrt{\log m}$ with probability tending to 1 as $m\rightarrow \infty.$
\begin{mythm}
\label{thm:concentration}
There exists a universal constant $C$ such that 
\[
P\left\{\|A-M\|_2 \geq \left(3\sqrt{2} \sqrt{\alpha + (k-1)\beta} + C\right)\sqrt{\log m} \right\} \rightarrow 0
\]
as $n \rightarrow \infty.$
\end{mythm}
\begin{proof}
The result is a direct consequence of Corollary 3.12 and Remark 3.13 of Bandeira \& Van Handel \cite{sharp_bounds} applied to $A-M$. Using the notation of that reference, $C$ is selected such that $C^2 > \tilde{c}_{1/2}$.
\qquad
\end{proof}

\begin{mycor}
\label{cor:proj_concentration}
There exists a constant depending only on $\alpha$ and $\beta$, denoted $C_{\alpha,\beta},$ such that
\[
P\left\{\|WW^T-VV^T\|_2 \geq \frac{C_{\alpha,\beta}}{\left(\alpha\left(1-1/m\right)-\beta\right)\sqrt{\log m}} \right\} \rightarrow 0
\]
as $n \rightarrow \infty.$
\end{mycor}
\begin{proof}
The result is obtained by combining Theorem~\ref{thm:concentration} with \eqref{eq:proj_error} and letting 
\[
C_{\alpha,\beta} = 3\sqrt{2} \sqrt{\alpha + (k-1)\beta} + C
\]
\qquad
\end{proof}

Corollary~\ref{cor:proj_concentration} gives an upper bound on $\|WW^T-VV^T\|_2$ that holds with high probability and decays asymptotically as $\frac{1}{\sqrt{\log m}}.$  It is important to note that this rate of decay is not sufficient to satisfy the assumptions required for Lemmas~\ref{lem:piv} and \ref{lem:piv_rand}.  Nevertheless, as evidenced by our numerical results our algorithms perform quite well in this sparse regime. In fact, under the assumption that our algorithm correctly identifies one point per cluster our results say that the root mean squared error $\frac{1}{n}\| W\widehat{\Pi} - VU \|_F$ decays like $\frac{1}{\sqrt{n \log n}}.$

In practice, Lemmas~\ref{lem:piv} and \ref{lem:piv_rand} are used to ensure that one column associated with each cluster is found. While matrices exist for which the worst-case CPQR bounds used are achieved, we conjecture that this does not occur for the structured matrices appearing in this application of CPQR. In fact, a stronger row wise error bound on $\|WQ - V\|_F$ for $Q$ that solves the orthogonal Procrustes problem would allow for the development of stronger results for our algorithm. We do not know of any sufficient results in the literature and are working to develop such bounds.

\subsection{Connections to OCS}
The preceding theory may also be connected to OCS. In particular, we may view the collection of orthogonal cone centers $\left\{q_j\right\}_{j=1}^k$ as an arbitrary rotation of the canonical basis vectors. This means that the indicator matrix $W$ is also relevant for the OCS model. However, now we need to assume that a fraction $1-\eta$ of the rows of $W$ are exactly zero, corresponding to nodes that do not belong to any cone. In practice, this motivates a potential modification of our algorithm to avoid assigning a cluster to any column of $\left\lvert U^TV_K^T \right\rvert$ where all of the entries are uniformly small.  

In the OCS framework, the prior theory carries through in a similar manner with a few small modifications we discuss here. In particular, Lemmas~\ref{lem:piv} and \ref{lem:piv_rand} must be modified to ensure that one column per cone is selected. We accomplish this in the following lemma by using the same conditioning argument as before and assuming that $\mu$ and $\delta$ are sufficiently small. 

\begin{mylem}
\label{lem:ocs_piv}
Let $V \in \R^{n\times k}$ have orthonormal columns with the columns $\{v_i\}_{i=1}^n$ of $V^T$ exhibiting OCS, specifically $\left\{v_i\right\}_{i=1}^n\in\ocs(\eta,\mu,\delta)$ with cone centers $\left\{q_i\right\}_{i=1}^k.$ Furthermore, let $c_M \equiv \max_{i\in\nset}\|v_i\|_2$ and assume that
\[
\delta c_M < \frac{2^{-k}}{\sqrt{n}} \quad \text{and}\quad\mu c_M^2 < \frac{2^{-2k-4}}{n}.
\]
If $V^T\Pi = QR$ is a CPQR factorization and $\C$ denotes the original indices of the first $k$ columns selected by the permutation $\Pi,$ then for each $j\in\kset$ the matrix $\left(V^T\right)_{:,\C}$ has exactly one column contained in $\K_{\mu}\left(q_j\right).$ 
\end{mylem}
\begin{proof}
As before, the smallest singular value of $\left(V^T\right)_{:,\C}$ is bounded from below by $\frac{2^{-k}}{\sqrt{n}}.$ This precludes any column of $\left(V^T\right)_{:,\C}$ from having norm less than $\frac{2^{-k}}{\sqrt{n}}$ since $c_M < 1.$ In particular, this implies that no column is outside the cones $\K_{\mu}\left(q_j\right)$ from the definition of OCS. 

It remains to prove that no two columns of $\left(V^T\right)_{:,\C}$ come from the same cone. We proceed by contradiction. Assume $i_1,i_2 \in \C$ and $v_{i_1},v_{i_2}\in\K_{\mu}(q_j)$ for some $j.$ Using our upper bound on $\mu c_M$ we may conclude that
\[
\|v_{i_1} - q_jq_j^Tv_{i_1}\|_2 \leq \frac{2^{-k-2}}{\sqrt{n}}
\]
and the same inequality holds for $v_{i_2}.$ Therefore, there exists a perturbation of $\left(V^T\right)_{:,\C}$ with $\ell_2$ norm bounded by $\frac{2^{-k}}{\sqrt{n}}$ that makes two columns co-linear, which contradicts our lower bound on the smallest singular value. Finally, since $\lvert\C\rvert = k$ the result follows.
\qquad 
\end{proof}

Asserting that we can select one point per cone via Lemma~\ref{lem:ocs_piv} yields the following recovery theorem.
\begin{mythm}
\label{thm:ocs}
Let $V \in \R^{n\times k}$ have orthonormal columns, and assume the columns of $V^T,$ denoted $\left\{v_i\right\}_{i=1}^n\in\ocs(\eta,\mu,\delta)$ with cone centers $\left\{q_i\right\}_{i=1}^k.$ Furthermore, let $c_M \equiv \max_{i\in\nset}\|v_i\|_2$ and assume that
\[
\delta c_M < \frac{2^{-k}}{\sqrt{n}} \quad \text{and}\quad\mu c_M^2 < \frac{2^{-2k-4}}{n}.
\]
If $U$ is computed by the deterministic algorithm in subsection \ref{sec:alg_sub} applied to $V^T$, then there exists a permutation matrix $\widehat{\Pi}$ such that $U$satisfies
\[
\frac{\|\widehat{\Pi}U^Tv_i - Q^Tv_i\|_2}{\|v_i\|_2} \leq \frac{(2+\sqrt{2})k^{3/2}c_M\mu}{c_m} \quad \forall \; i\in\left[n\right]
\]
where the columns of $Q$ are the underlying cone centers $\left\{q_i\right\}_{i=1}^k$ and $c_m\equiv\min_{j\in\C}\|v_j\|_2$.
\end{mythm} 
\begin{proof}
Based on Lemma~\ref{lem:ocs_piv} there exists a permutation $\widehat{\Pi}$ such that
\[
\|\left(V^T\right)_{:,\C}\widehat{\Pi} - QD\|_F \leq \sqrt{2k}c_M\mu,  
\]
where $D$ is a diagonal matrix and $D_{ii}$ is equal to the two norm of the $i^{\text{th}}$ column of $\left(V^T\right)_{:,\C}.$ Equivalently we may assert that
\[
\left\|\left(V^T\right)_{:,\C} - Q\widehat{\Pi}^T\widehat{\Pi}D\widehat{\Pi}^T\right\|_F \leq \sqrt{2k}c_M\mu.
\] 
Thinking of $\left(Q\widehat{\Pi}^T\right)\left(\widehat{\Pi}D\widehat{\Pi}^T\right)$ as the polar factorization of some matrix divided into its orthogonal and positive definite parts, we use perturbation theory for polar factorizations \cite{Higham_polar} to bound the distance between $U$ and $Q\widehat{\Pi}^T$ as 
\[
\|U-Q\widehat{\Pi}^T\|_F \leq \frac{(2+\sqrt{2})k^{3/2}c_M\mu}{c_m}.
\]
Using 
\[
\|\widehat{\Pi}U^Tv_i - Q^Tv_i\|_2 \leq \|U-Q\widehat{\Pi}\|_F \|v_i\|_2
\]   
allows us to conclude the desired result.
\qquad
\end{proof}

This result implies that if $\mu$ and $\delta$ are small enough our algorithm can correctly assign the $\eta n$ nodes that belong to clusters, up to an arbitrary labeling encoded in $\widehat{\Pi}.$ While our CPQR analysis yields a lower bound on $\min_{j\in\C}\|v_j\|_2$ in terms of $k$ and $n$ it is often pessimistic in practice, so we leave our result in terms of the smallest column norm. Adding some assumptions about the total mass of nodes per cone allows Lemma~\ref{lem:ocs_piv} and Theorem~\ref{thm:ocs} to be modified in a similar manner to address the randomized variant of our algorithm.

Ultimately, these results are somewhat unwieldy and appear pessimistic when compared with the observed performance of our algorithm. It is also difficult to assert that a graph's eigenvectors will obey such structure (especially if we require small $\delta$ and $\mu$.) Later, we demonstrate the behavior of our algorithm on real graphs and its favorable performance when compared with standard spectral clustering methods.

\subsection{Connected components}

A limiting, and somewhat cleaner, scenario of the graph clustering problem is the graph partitioning problem. Here, we are given a graph that is comprised of $k$ disjoint connected components. A relatively simple problem is to partition the graph into these $k$ components. It turns out, given any orthonormal basis for the $k$-dimensional eigenspace associated with the zero eigenvalue of the normalized Laplacian our algorithm exactly recovers the partition (with high probability in the case of the randomized variant). While there are many algorithms for finding connected components, we find this to be an interesting property of our algorithm.   

Given a matrix $W\in\R^{n\times k}$ with orthonormal columns and at most one nonzero in each row, let $V\in\R^{n\times k}$ differ from $W$ by some rotation $Z\in\O_k$.   In this case, it is simple to modify the analysis of subsection~\ref{subsec:analysis} to show that our algorithm applied to $V$ will exactly recover $W$ up to a permutation of the columns. 

A more interesting observation is that given any $\tilde{k} < k$ dimensional subspace of the $k$-dimensional subspace associated with the zero eigenvalue of the normalized Laplacian our deterministic algorithm will partition the nodes into $\tilde{k}$ connected components. We prove this by showing that our algorithm necessarily places each connected component in the same cluster by virtue of its assignment step.
\begin{mythm}
\label{thm:connected_comp}
Let $V\in \R^{n\times \tilde{k}}$ have orthonormal columns, let $W\in R^{n\times k}$ have orthonormal columns with exactly one nonzero per row, and assume $\tilde{k}\leq k.$ If $V = WZ$ for some $Z\in \R^{k \times \tilde{k}}$ with orthonormal columns, then our deterministic algorithm partitions $\nset$ into $\tilde{k}$ clusters in a manner such that no two clusters contain rows of $W$ with the same sparsity pattern. 
\end{mythm}
\begin{proof}
Let $\mathcal{S}_i \subset \nset$ denote the support of column $i$ of $W.$ Using this notation, we see that the columns of $\left(V^T\right)_{:,\mathcal{S}_i}$ are co-linear \textemdash they are all proportional to the $i^{\text{th}}$ column of $Z^T.$ 
This implies that given any orthogonal matrix $U,$ each column of $\left\lvert U^T\left(V^T\right)_{:,\mathcal{S}_i} \right\rvert$ attains its maximum in the same row. Therefore, we see that each connected component is assigned to a single cluster using our algorithm

\qquad
\end{proof}
\begin{myrem}
In general the use of the CPQR should ensure that all of the output clusters will be nonempty, however, this depends on properties of $Z.$ Intuitively, because $U$ is the closest orthogonal matrix to $\tilde{k}$ different well-conditioned columns of $Z^T$ at least one column of $Z^T$ will align closely with one column of $U.$  
\end{myrem}

Interestingly, this property is not shared by the \texttt{k-means} algorithm in the presence of degree heterogeneity. In fact, given an initial guess that properly partitions the nodes into disconnected sets, \texttt{k-means} may move to a local minimum of its objective function that does not respect the proper partition. Intuitively, this may occur when it is advantageous to place a cluster with a center near the origin that collects low degree nodes irrespective of cluster. While a simple thought experiment, it provides additional justification for using angular relationships between embedded nodes rather than a distance-based clustering.

\section{Numerical results}
We now present numerical simulations to validate the performance of our algorithm through examination of its behavior for multi-way spectral clustering. All of our experiments were conducted using MATLAB\textsuperscript{\textregistered} and the included \texttt{k-means++} implementation was used as a point of comparison. If \texttt{k-means} failed to converge after 100 iterations (the default maximum) we simply took the output as-is. Code implementing our methodology and these experiments may be found at \url{https://github.com/asdamle/QR-spectral-clustering}.

\subsection{SBM}
We first consider the SBM with $k=9$ equisized clusters and show via simulation that our algorithm, in contrast to \texttt{k-means++}, recovers the phase transition behavior near the information theoretic threshold \cite{abbe,abbe2015community,hajek2015achieving,agarwal2015multisection}. We also demonstrate similar phase transition behavior when there are $k=7$ unequisized clusters.

For our tests with the SBM we compare the behavior of three different clustering schemes: our algorithm,  \texttt{k-means} using the columns $\C$ from our algorithm as an initial guess for cluster centers, and \texttt{k-means++} itself. We exclusively utilize the randomized variant of our algorithm with an oversampling factor of $\gamma = 5.$ For more general problems, the non-random method may perform better, however, in this context they performed almost identically and hence the deterministic results are omitted.

Interestingly, in the regime where $p$ and $q$ are both $\Theta(\log n / n)$ we do not necessarily expect rapid enough asymptotic concentration of $A$ for our results to theoretically guarantee recovery to the information theoretic limit. Nevertheless, we do observe good behavior of our algorithm when applied in this setting and, as we will see, the same behavior is not observed with \texttt{k-means++}. 

One major avenue of work in recent years has been in the area of semidefinite programming (SDP) relaxation for clustering \cite{hajek2016achieving,abbe2015community,abbe}.  Broadly speaking, such relaxations recast \eqref{eq:combo} in terms of $XX^T$ and then relax $XX^T$ to a semidefinite matrix $Z$. These SDP relaxations often enjoy strong consistency results on recovery down to the information theoretic limit in the case of the SBM, in which setting the optimal solution $Z_*$ can be used to recover the true clusters exactly with high probability. However, these algorithms become computationally intractable as the size of the graph grows.

\subsubsection{Equisized clusters}\label{sec:equisized}
First, we consider $k=9$ equisized clusters each with size $m=150$ (\emph{i.e.}, $n=1350$) and within-cluster and between-cluster connection probabilities ${p = \alpha \log m / m}$ and ${q = \beta \log m / m}$, respectively. Defining a grid of 80 equispaced values of $\alpha$ and 40 equispaced values of $\beta$, we generated 50 instances of the SBM for each $(\alpha,\beta)$ pair (redrawing the adjacency matrix if it was disconnected).  Then, for both the adjacency matrix $A$ and the degree-normalized adjacency matrix $A_N$ with entries $(A_N)_{ij}\equiv A_{ij} / \sqrt{d_id_j}$, where $d_i$ is the degree of node $i$, we computed the top nine eigenvectors and used them as input to each of the three algorithms mentioned before.  While for degree-regular graphs we do not expect to observe differences in the results between $A$ and $A_N$, when applied to graphs not coming from the SBM or that are not as degree-regular we anticipate better performance from degree-normalization.

In the plots of Figure~\ref{fig:9block} we color each point according to the fraction of trials resulting in successful recovery for each $(\alpha,\beta)$ pair, \ie, the fraction of trials in which all nodes were correctly clustered. These phase diagrams show the primary advantage of our algorithm over \texttt{k-means++}: robust recovery (where possible). Whether used to explicitly compute the clustering directly, or to seed \texttt{k-means}, we cleanly recover a sharp phase transition that hews to the theory line. In contrast, \texttt{k-means++} fails to exactly recover the clusters a substantial portion of the time, even far from the phase transition.

\begin{figure}[ht!]
\centering
\includegraphics[width=.45\linewidth]{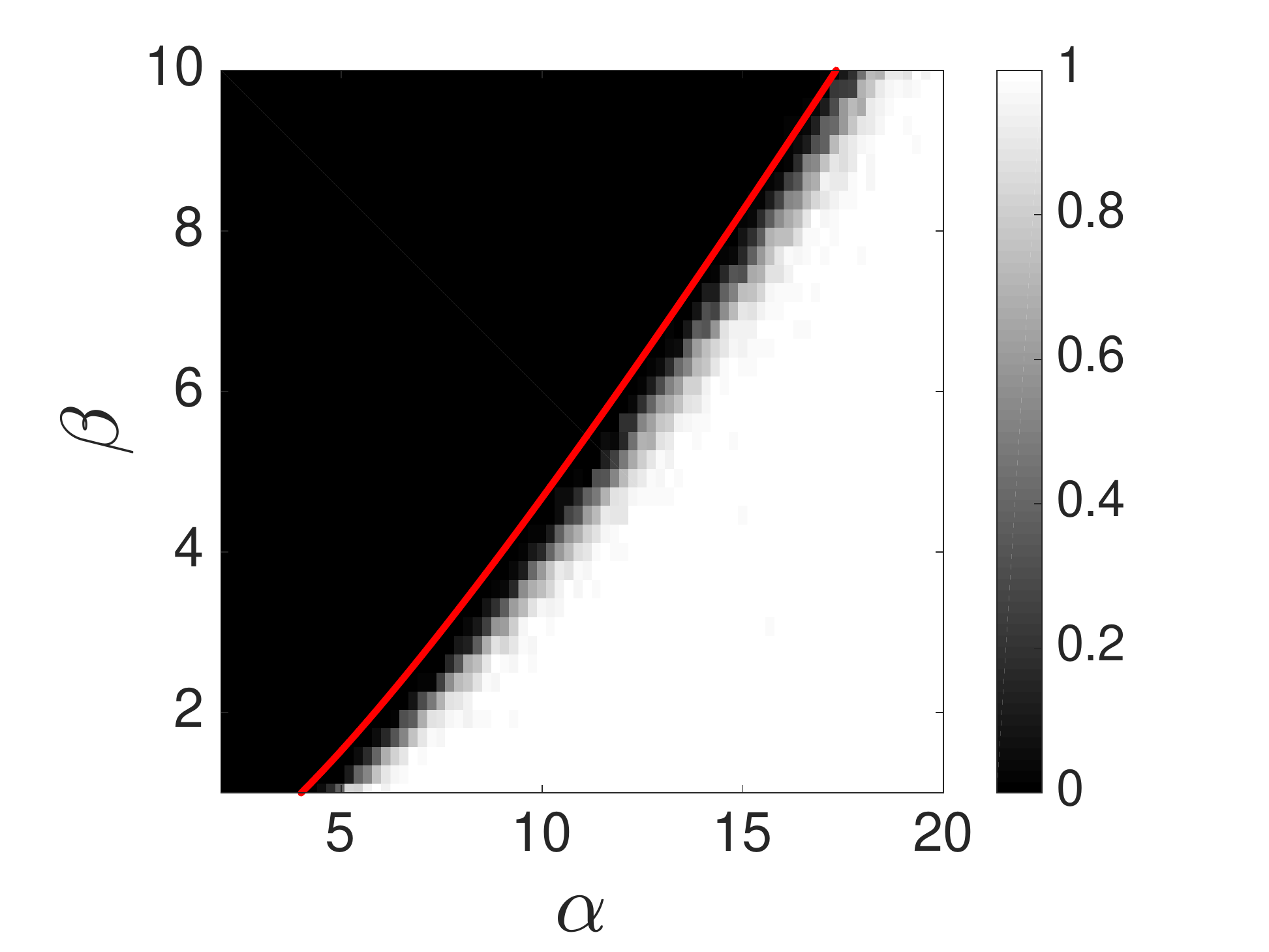}
\includegraphics[width=.45\linewidth]{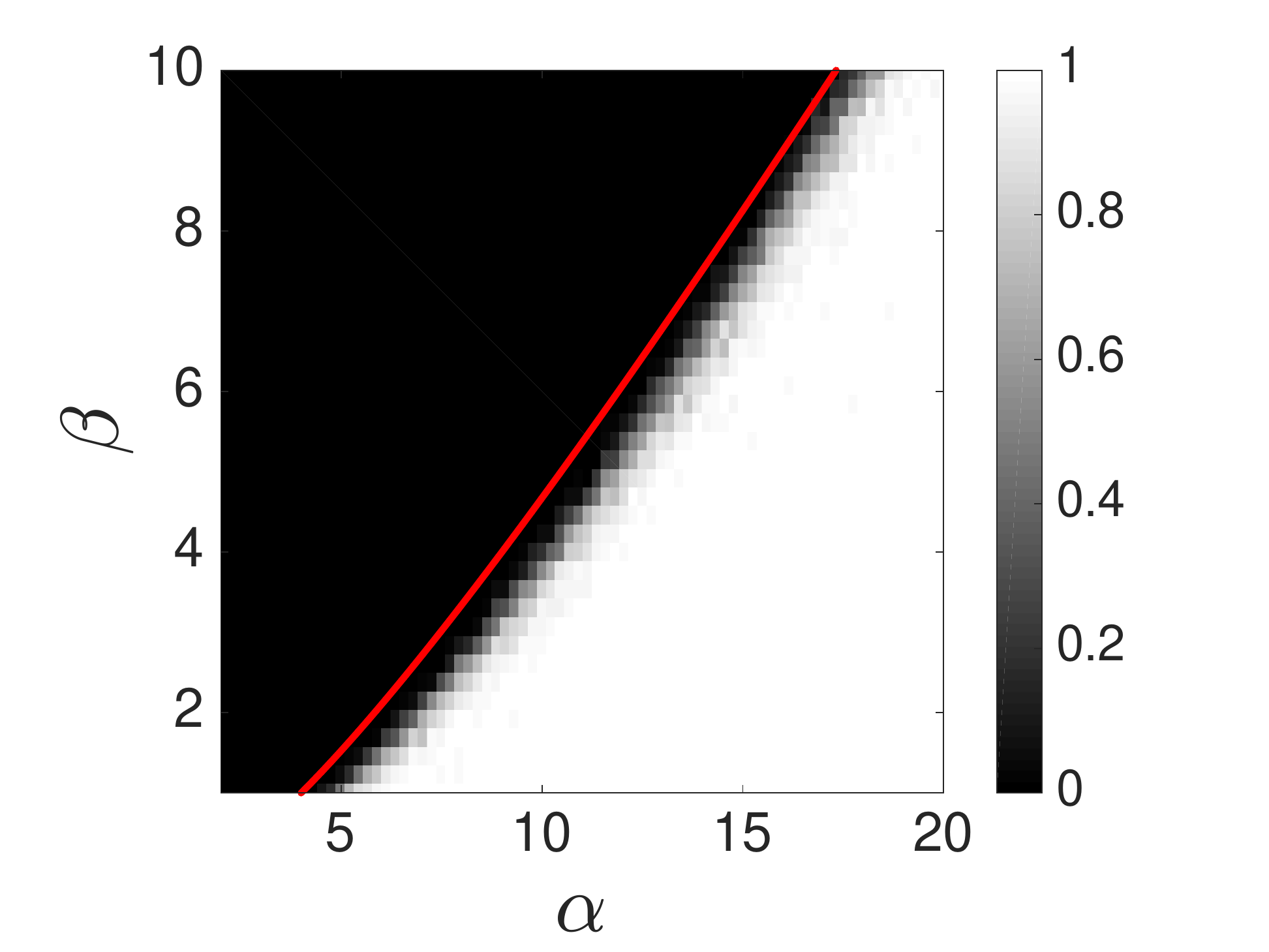}
\includegraphics[width=.45\linewidth]{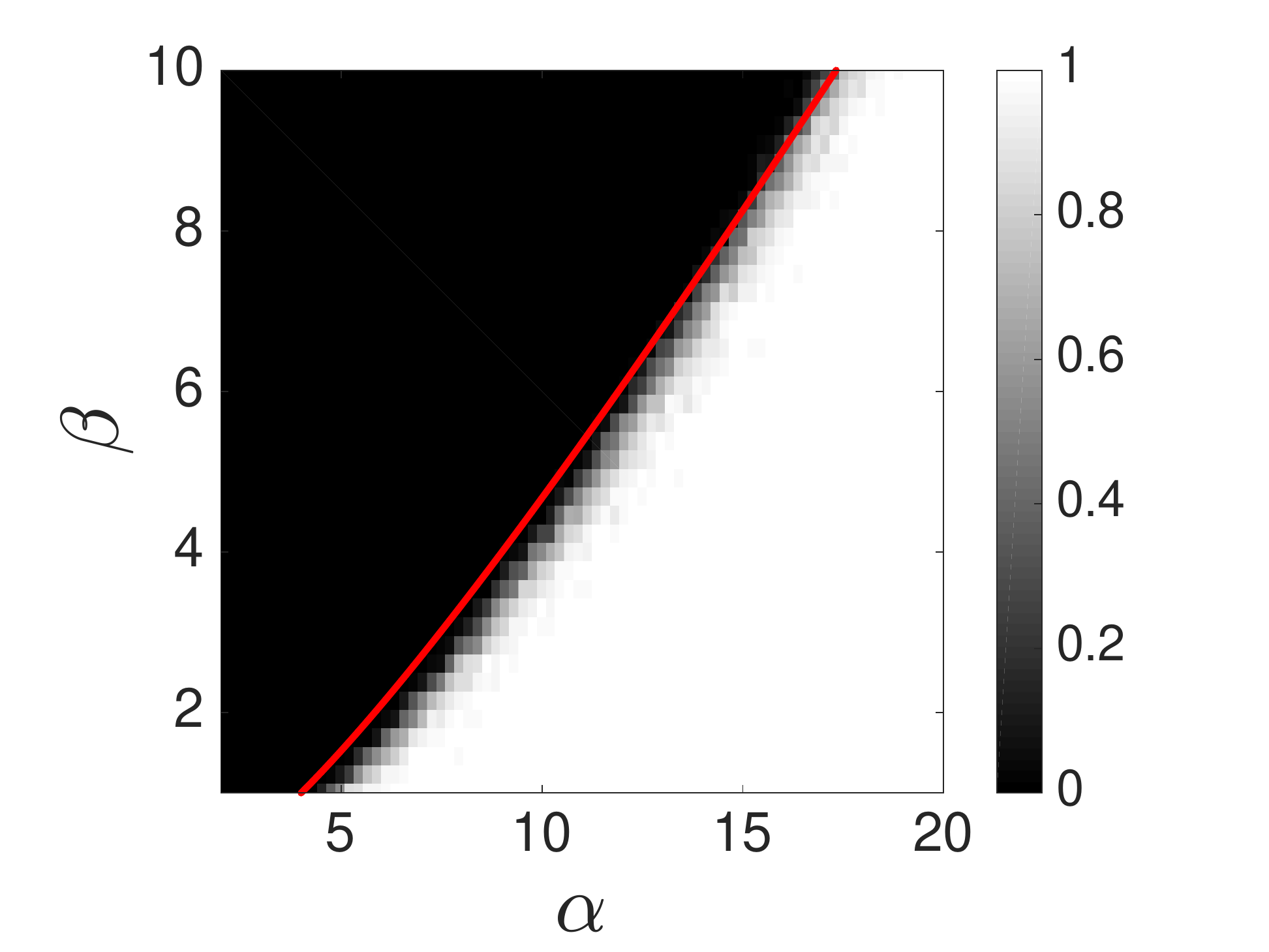}
\includegraphics[width=.45\linewidth]{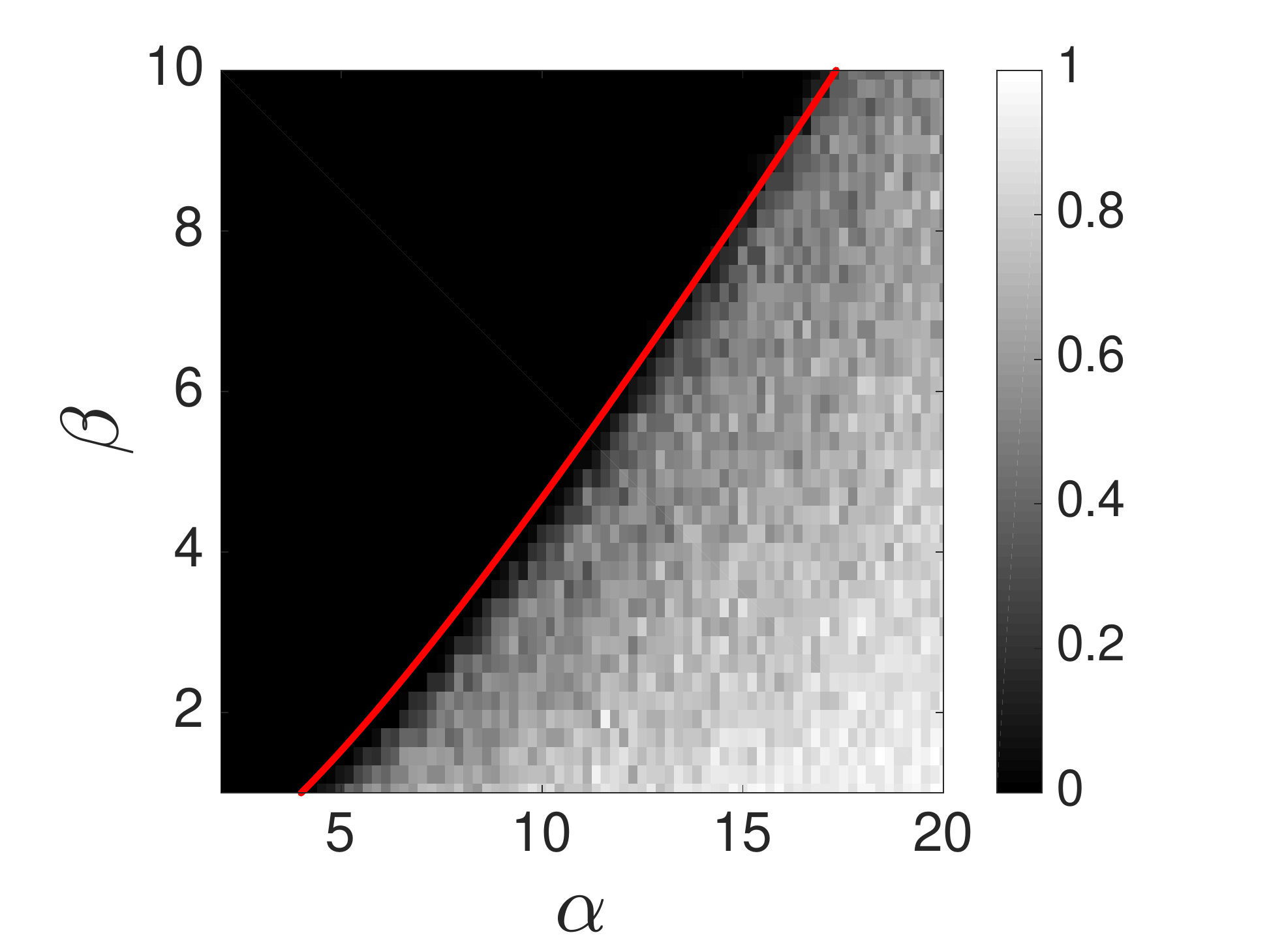}
\caption{The phase plots here show the fraction of trials exhibiting exact recovery in the case of equisized clusters, and the asymptotic transition curve $\sqrt{\alpha} - \sqrt{\beta} = 1$ is given in red. The top row of phase plots corresponds to our randomized algorithm applied to eigenvectors $V_k$ of $A$ (top-left) and of $A_N$ (top-right).  We also give results based on seeding \texttt{k-means} with the clustering $\C$ from our algorithm applied to eigenvectors from $A_N$ (bottom-left), and results for vanilla \texttt{k-means++} (bottom-right).}
\label{fig:9block}
\end{figure}

Theoretically, if we scale $k$ as $k=o(\log m)$ and take the asymptotic limit as $m\to\infty$ then a sharp phase transition between where recovery is possible almost surely and where it is not occurs at the threshold curve $\sqrt{\alpha} - \sqrt{\beta} = 1$ \cite{agarwal2015multisection}. (Note that similar results exist elsewhere \cite{hajek2015achieving}, albeit with a slightly different definition of $\alpha$ and $\beta$.) In Figure~\ref{fig:9block} we observe that our phase transition does not quite match the asymptotically expected behavior. However, based on results for the two-block case \cite{abbe}, various arguments \cite{agarwal2015multisection} show that we expect the location of the recovery threshold in the finite case to deviate from the asymptotic threshold with error that decays only very slowly, \ie, $\Omega(1/\log m)$. 

To further explore the discrepancy between our algorithms performance and the asymptotic theory we consider a slightly different criteria than exact recovery. A natural metric for a good graph clustering is to minimize the number of edges between clusters while normalizing for cluster size. Specifically, we define the multi-way cut metric for a $k$-way partition of nodes into nonempty sets $S_1,S_2,\ldots,S_k$ as
\begin{equation}
\label{eq:cut_metric}
\max_{i=1,\ldots,k} \frac{\#\{\text{Edges between $S_i$ and $\overline{S_i}$}\}}{\lvert S_i \rvert}.
\end{equation} 
Figure~\ref{fig:comp_true} shows the fraction of times that our algorithm yields as good or better a multi-way cut metric than the true clustering. Near the threshold we may not be recovering the underlying clustering, but perhaps we should not expect to since we are often finding a better clustering under a slightly different metric. 

\begin{figure}[ht!]
\centering
\includegraphics[width=.5\linewidth]{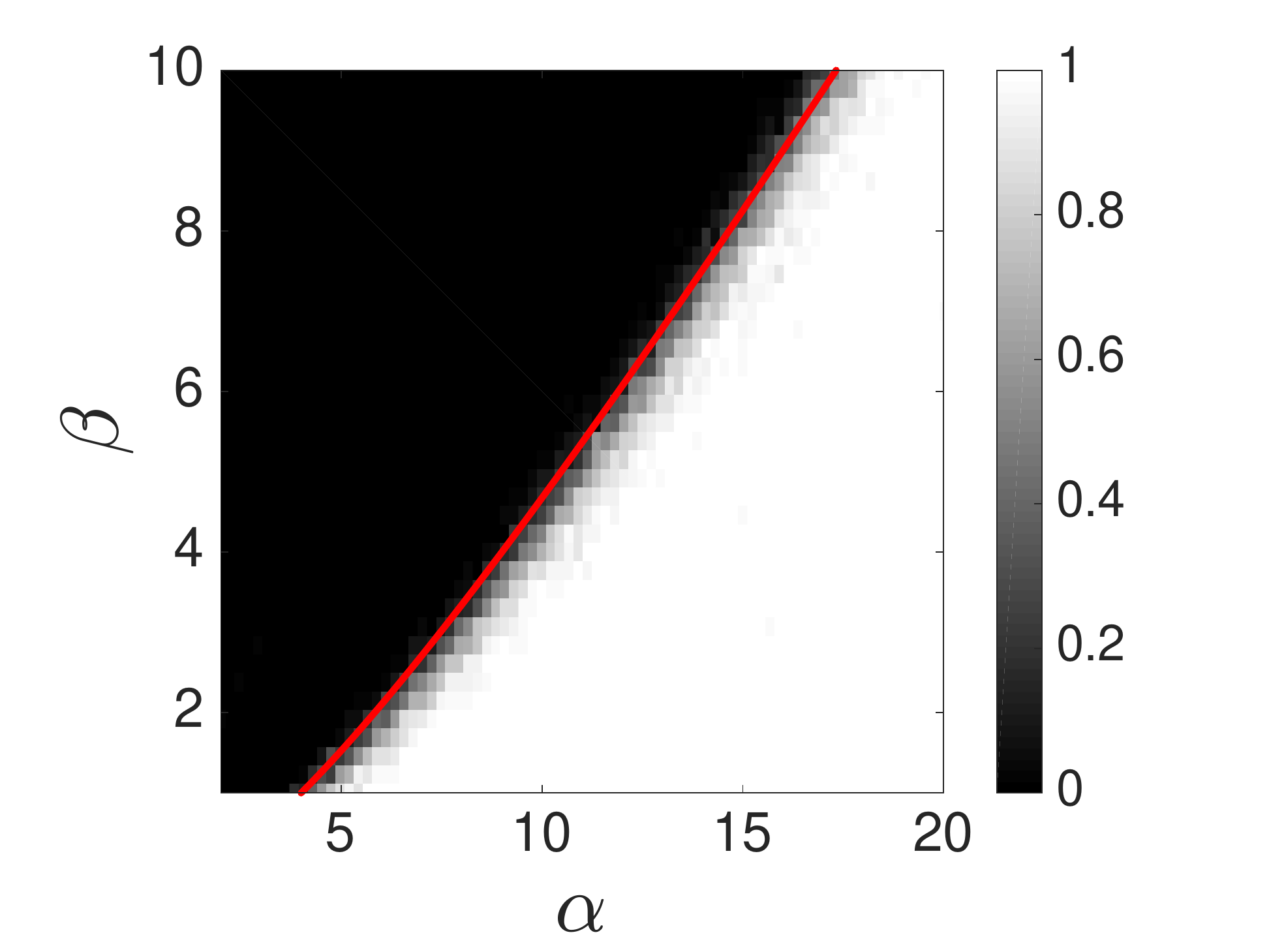}
\caption{The phase plot shows the fraction of trials (for equisized clusters) where application of our algorithm to eigenvectors $V_k$ of $A$ yields a multi-way cut metric \eqref{eq:cut_metric} less than or equal to that of the true underlying clusters. The asymptotic transition curve $\sqrt{\alpha} - \sqrt{\beta} = 1$ is given in red.}
\label{fig:comp_true}
\end{figure}

\subsubsection{Unequisized clusters}
Secondly, we consider the case where there are $k=7$ \emph{unequisized} clusters with sizes $m = 70, 80, 90, 100, 110, 120,$ and $130$. In this case we define the within- and between-cluster connection probabilities as ${p = \alpha \log 70 /70}$ and ${q = \beta \log 70 /70}$. As before, we test the algorithms for exact recovery and present the phase plots in Figure~\ref{fig:7block_unequal}. As before, our algorithm shows a sharp phase transition whereas \texttt{k-means++} does not. Such behavior is still expected \cite{abbe2015community}, though the characterization of its location becomes significantly more complicated. Importantly, our algorithm seamlessly deals with unknown cluster sizes.

\begin{figure}[ht!]
\centering
\includegraphics[width=.45\linewidth]{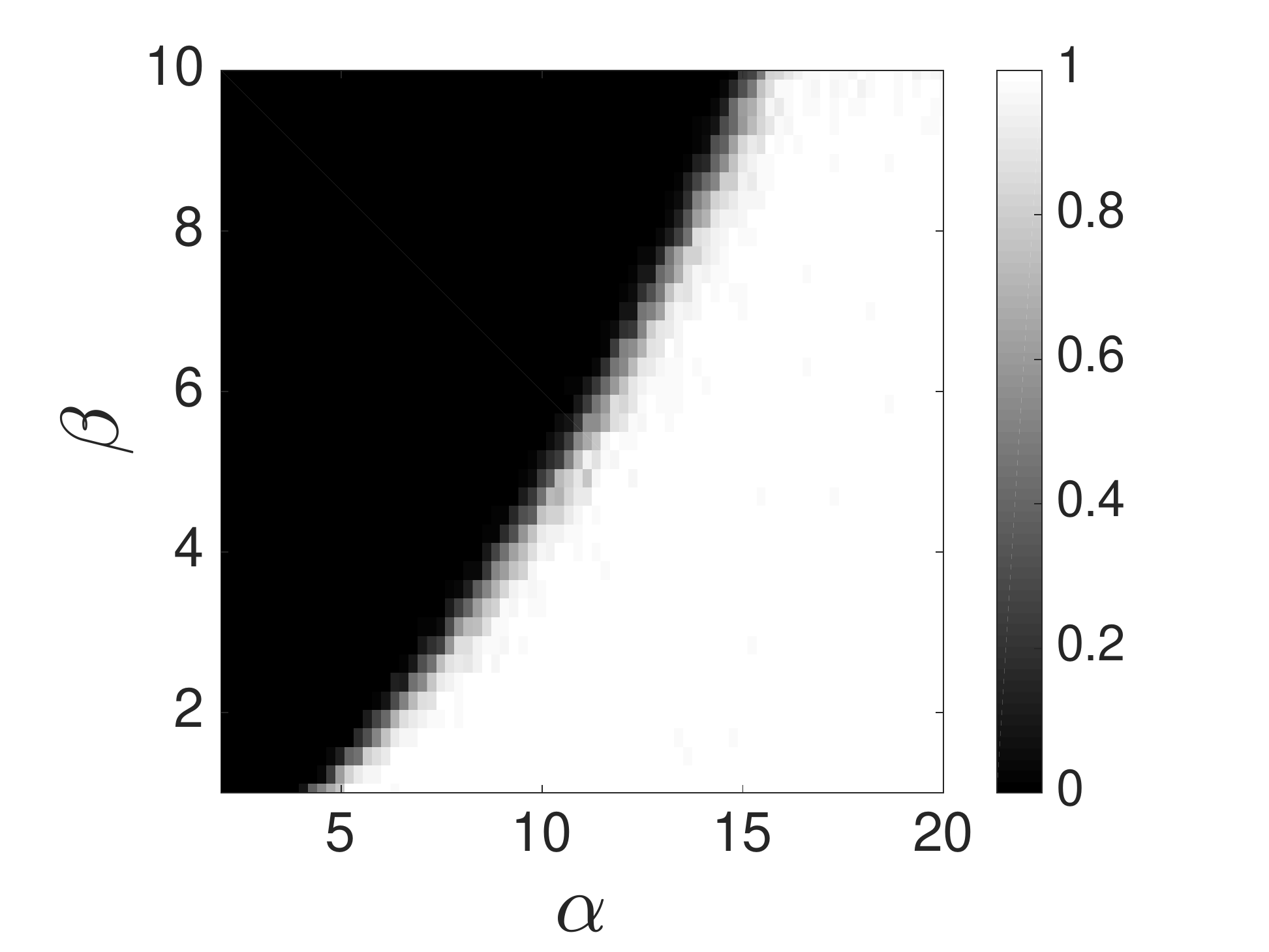}
\includegraphics[width=.45\linewidth]{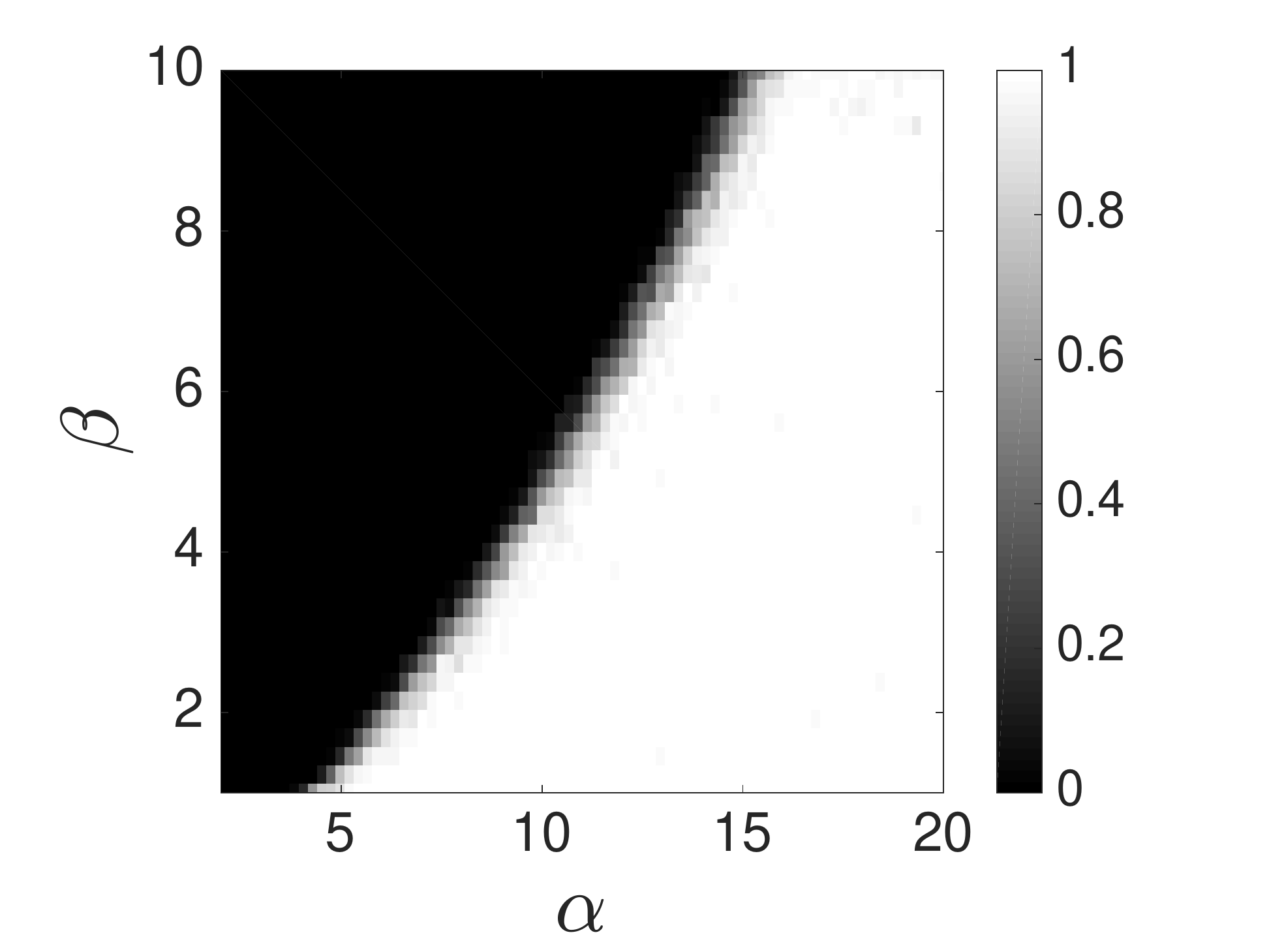}
\includegraphics[width=.45\linewidth]{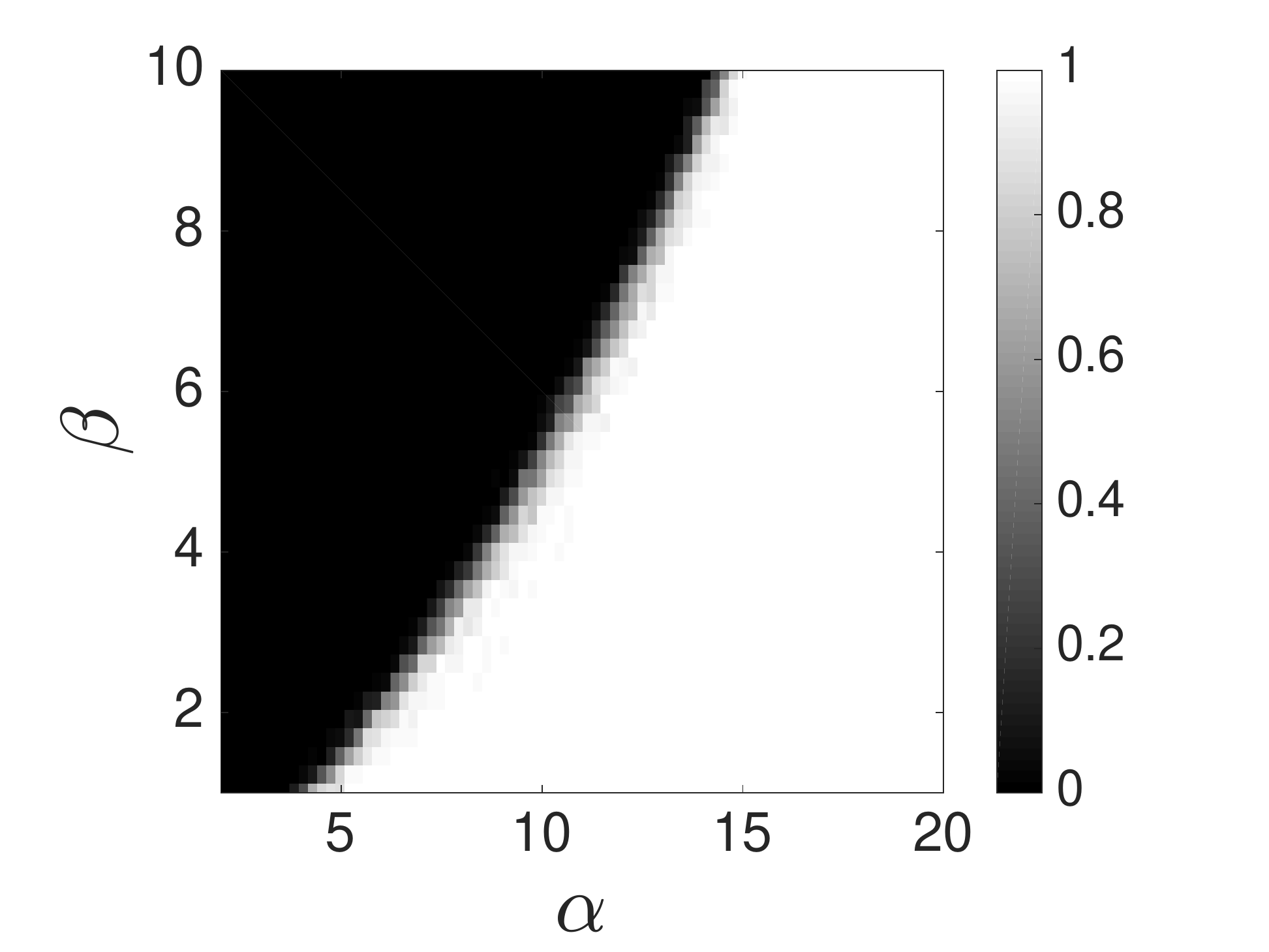}
\includegraphics[width=.45\linewidth]{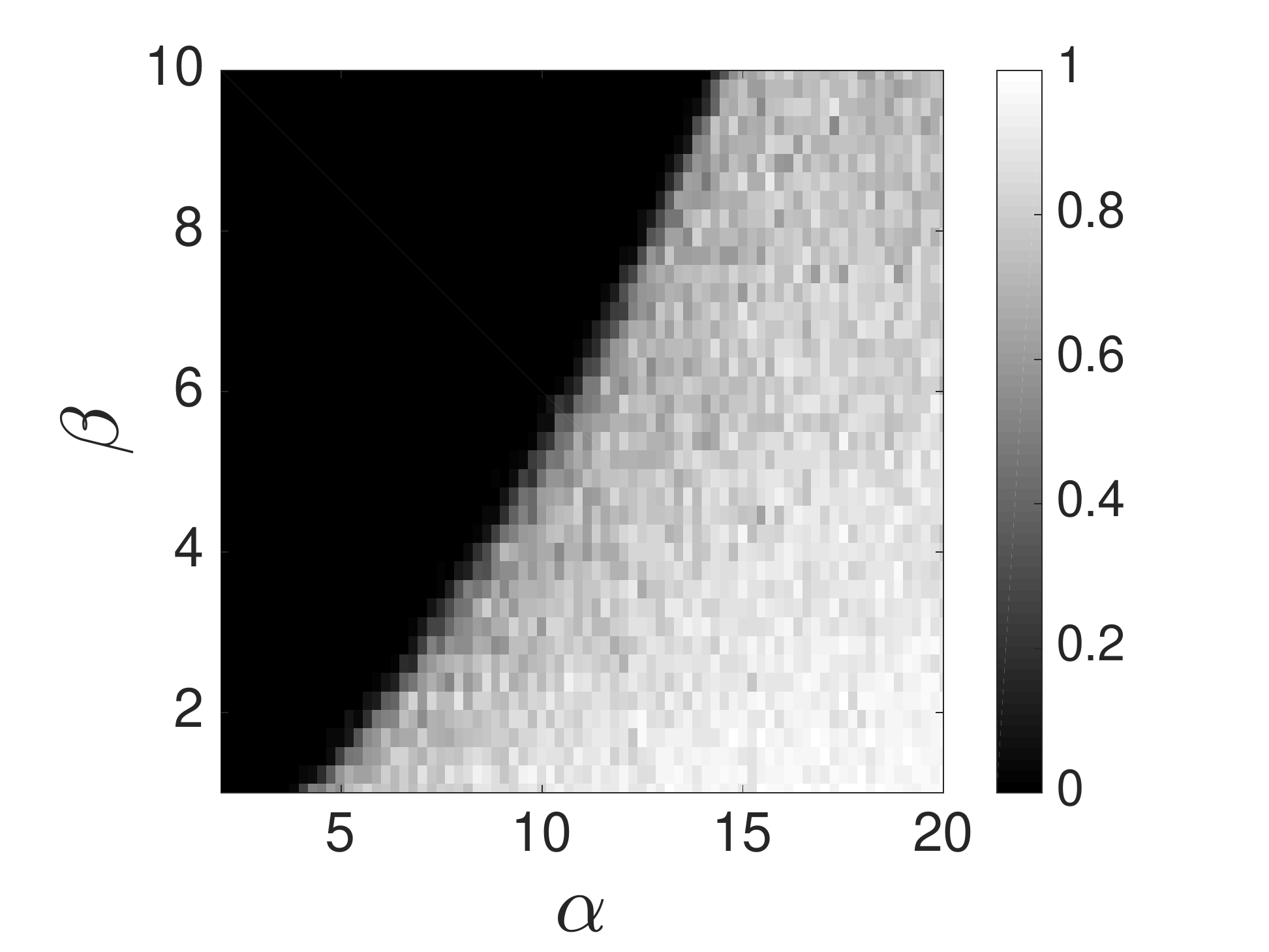}
\caption{The phase plots here show the fraction of trials exhibiting exact recovery in the case of unequisized clusters. The top row of phase plots corresponds to our randomized algorithm applied to eigenvectors $V_k$ of $A$ (top-left) and of $A_N$ (top-right).  We also give results-based on seeding \texttt{k-means} with the clustering $\C$ from our algorithm applied to eigenvectors from $A_N$ (bottom-left), and results for vanilla \texttt{k-means++} (bottom-right).}
\label{fig:7block_unequal}
\end{figure}

\subsection{A real world graph}
We complement our experiments using the SBM by demonstrating the application of our algorithm to a real world graph and comparing its performance with that of \texttt{k-means++}. To measure performance we consider both the \texttt{k-means} objective function \textemdash the sum of squared distances from points to their cluster centers \textemdash and the multi-way cut metric \eqref{eq:cut_metric}. Since it is a metric on the graph itself (as opposed to the embedding), minimization of the multi-way cut metric over possible partitions is the more interpretable metric. Furthermore, as we will observe, a smaller \texttt{k-means} objective value does not necessarily translate into a smaller multi-way cut metric.

The graph we consider is the collaboration network for the arXiv Astrophysics category \cite{snapnets}. This graph consists of 18,772 nodes and 198,110 undirected edges corresponding to authors and co-authorship, respectively. The graph is partitioned into 290 connected components, the largest of which contains 17,903 nodes. For all of the experiments conducted here, we used $k$ eigenvectors corresponding to the $k$ largest eigenvalues of a degree-normalized adjacency matrix $A_N$ as in subsection \ref{sec:equisized}.

We begin with a simple experiment conducted using the whole graph. Because the graph has 290 connected components, it is possible to partition the nodes into $k$ disconnected subgraphs using any subspace of the invariant space associated with the first 290 eigenvalues. As outlined in Theorem~\ref{thm:connected_comp} our algorithm theoretically accomplishes this task without fail. Our experiments validate this fact, looking for a $10$-way partition our deterministic algorithm achieved a multi-way cut metric of zero. In contrast, over 50 trials the smallest multi-way cut metric found using \texttt{k-means++} was $0.08.$ Seeding \texttt{k-means} using our clustering resulted in multi-way cut metric of $11.03,$ though the \texttt{k-means} objective was decreased from $6.31$ to $4.69.$ In particular this demonstrates that the \texttt{k-means} objective function is not necessarily a good proxy for the multi-way cut metric.

Next, we take the largest connected component of the graph and seek to partition it six ways. Admittedly, as with many real world graphs it is not an easy problem to determine a value for $k$ that leads to a natural clustering. Nevertheless, here we choose $k=6$ and note that there is a slightly larger gap in the spectrum between the sixth and seventh eigenvalues than others in the surrounding area. 

Table~\ref{tab:graph} summarizes the results comparing our deterministic algorithm with 50 trials of \texttt{k-means++}. While our algorithm does not find the best possible multi-way cut, it gets close in a deterministic manner and serves as a very good seeding for \texttt{k-means}. Interestingly, this is accomplished with complete disregard for the \texttt{k-means} objective function: our algorithm results in a sum of squared distances larger than any of the local minima found by \texttt{k-means++}. However, by seeding \texttt{k-means} with the clustering performed by our algorithm we find as good a local minima of the \texttt{k-means} objective as any found by \texttt{k-means++}. 
\begin{table}
{\renewcommand{\arraystretch}{1.2}

\colorlet{tableheadcolor}{gray!50} 
\newcommand{\headcol}{\rowcolor{tableheadcolor}} 
\colorlet{tablerowcolor}{gray!10} 
\newcommand{\rowcol}{\rowcolor{tablerowcolor}} %
\newcommand{\topline}{\arrayrulecolor{black}\specialrule{0.1em}{\abovetopsep}{0.5pt}%
            \arrayrulecolor{tableheadcolor}\specialrule{\belowrulesep}{0pt}{-3pt}%
            \arrayrulecolor{black}
            }
\newcommand{\midline}{\arrayrulecolor{tableheadcolor}\specialrule{\aboverulesep}{-1pt}{0pt}%
            \arrayrulecolor{black}\specialrule{\lightrulewidth}{0pt}{0pt}%
            \arrayrulecolor{white}\specialrule{\belowrulesep}{0pt}{-3pt}%
            \arrayrulecolor{black}
            }
\newcommand{\rowmidlinecw}{\arrayrulecolor{tablerowcolor}\specialrule{\aboverulesep}{0pt}{0pt}%
            \arrayrulecolor{black}\specialrule{\lightrulewidth}{0pt}{0pt}%
            \arrayrulecolor{white}\specialrule{\belowrulesep}{0pt}{0pt}%
         \arrayrulecolor{black}}
\newcommand{\rowmidlinewc}{\arrayrulecolor{white}\specialrule{\aboverulesep}{0pt}{0pt}%
            \arrayrulecolor{black}\specialrule{\lightrulewidth}{0pt}{0pt}%
            \arrayrulecolor{tablerowcolor}\specialrule{\belowrulesep}{0pt}{0pt}%
            \arrayrulecolor{black}}
\newcommand{\rowmidlinew}{\arrayrulecolor{white}\specialrule{\aboverulesep}{0pt}{0pt}%
            \arrayrulecolor{black}}
\newcommand{\rowmidlinec}{\arrayrulecolor{tablerowcolor}\specialrule{\aboverulesep}{0pt}{0pt}%
            \arrayrulecolor{black}}
\newcommand{\bottomline}{\arrayrulecolor{white}\specialrule{\aboverulesep}{0pt}{-2pt}%
            \arrayrulecolor{black}\specialrule{\heavyrulewidth}{0pt}{\belowbottomsep}}%
\newcommand{\bottomlinec}{\arrayrulecolor{tablerowcolor}\specialrule{\aboverulesep}{0pt}{0pt}%
            \arrayrulecolor{black}\specialrule{\heavyrulewidth}{0pt}{\belowbottomsep}}%

\caption{Comparison of deterministic CPQR-based clustering and \texttt{k-means++}.}
\label{tab:graph}
\rowcolors{2}{gray!25}{white}
\centering
\begin{tabular}{lcc} \topline\rowcolor{gray!50}
 { Algorithm} & { \texttt{k-means} objective} & { multi-way cut~\eqref{eq:cut_metric} } \\ \midline
 \texttt{k-means++} mean & 1.36 & 8.48 \\ 
 \texttt{k-means++} median & 1.46 & 10.21  \\ 
 \texttt{k-means++} minimum & 0.76 & 1.86 \\ 
 \texttt{k-means++} maximum & 2.52 & 42.03 \\ 
 CPQR-based algorithm & 2.52 & 1.92 \\ 
 \texttt{k-means} seeded with our algorithm & 0.76  & 1.86 \\ \bottomline
\end{tabular}
}
\end{table}

\section{Discussion and conclusion}
We have presented a new efficient (particularly the randomized variant) algorithm for spectral clustering of graphs with community structure. In contrast to the traditionally used \texttt{k-means} algorithm, our method requires no initial guess for cluster centers and achieves the theoretically expected recovery results for the SBM. Given that a bad initial guess can mean the \texttt{k-means} algorithm does not achieve the desired result this is a particularly important feature. Furthermore, we can always use our algorithm to generate an initial seed for \texttt{k-means} and observe in our experiments that this can provide small gains in recovery near the phase transition boundary. When considering a real world graph, our algorithm compares favorably as a means for determining clusters that achieve a small multi-way cut metric.

Recent results have yielded important understanding of what can be possible for recovery in the SBM. However, the SDP-based methods that achieve these results do not scale to large problems of practical interest. Conversely, the traditionally used \texttt{k-means} algorithm, while scalable to large problems, fails to achieve the best possible behavior on these small scale SBMs due to its dependence on an initial guess. Our algorithm is both scalable to large problems and matches the behavior of SDPs on smaller model problems. These two properties make it attractive for general use.

Here we have explored the behavior of our algorithm on the SBM and provided theoretical justification for its use. To motivate its use for more general problems, we have discussed its connections to the more broadly applicable OCS of clusters arising naturally in various problems. Furthermore, Theorems~\ref{thm:main} and~\ref{thm:main_rand} may be extensible to cluster indicators with a more general structure, though at the expense of weaker results. We intend to further explore the behavior of our algorithms on a wider range of real-world graphs in future work.

\section*{Funding}
This work was supported by the National Science Foundation [DMS-1606277 to A.D.]; Stanford [Stanford Graduate Fellowship to V.M.]; and the United States Department of Energy [DE-FG02-97ER25308 to V.M., DE-FC02-13ER26134 to L.Y., and DE-SC0009409 to L.Y.].

\section*{Acknowledgment}
The authors thank Austin Benson, Sven Schmit, Nolan Skochdopole, and Yuekai Sun for useful discussion, as well as Stanford University and the Stanford Research Computing Center for providing computational resources and support that have contributed to these research results.

\bibliographystyle{plain}
\bibliography{clustering}

\end{document}